\documentclass{article}
\usepackage{amsmath,amssymb,amsthm}
\newtheorem{theorem}{Theorem}
\newtheorem{lemma}{Lemma}
\newtheorem{corollary}{Corollary}
\newtheorem{example}{Example}
\begin{document}
\title{Functionals of spatial point processes having a density with respect to the Poisson process}
\author{Viktor Bene\v{s}$^1$, Mark\'{e}ta Zikmundov\'{a}$^2$}

\maketitle
\noindent $^1$Charles University in Prague, Faculty of Mathematics and Physics, Department of Probability and Mathematical Statistics, Sokolovsk\'{a} 83, 18675 Praha 8-Karl\'{\i}n, Czech Republic, benesv@karlin.mff.cuni.cz

\noindent $^2$Institute of Chemical Technology, Faculty of Chemical Engineering, Department of Mathematics, Technick\'{a} 5, 16628 Praha 6-Dejvice, Czech Republic, zikmundm@vscht.cz
%\footnote{\noindent Research supported by the Czech Science Foundation, grant P201-10-0472}

\bigskip

\noindent {\bf Abstract}

\noindent $U$-statistics of spatial point processes given by a density with respect to a Poisson process are investigated. In the first half of the paper general relations are derived for the moments of the functionals using kernels from the Wiener-It\^o chaos expansion. In the second half we obtain more explicit results for a system of $U$-statistics of some parametric models in stochastic geometry. In the logaritmic form functionals are connected to Gibbs models. There is an inequality between moments of Poisson and non-Poisson functionals in this case, and we have a version of the central limit theorem in the Poisson case.

\bigskip

\noindent Keywords: difference of a functional, limit theorem, moments, U-statistics

\medskip

\noindent Classification: 60G55, 60D05

\section{Introduction}
Recently the investigation of functionals of Poisson point processes using differences and Wiener-It\^o chaos expansion has been developed, cf. \cite{RefJ}. In \cite{RefB} central limit theorems for $U$-statistics of Poisson processes were derived based on Malliavin calculus and the Stein method. The Wiener chaos theory involves both Gaussian and Poisson multiple integrals \cite{RefT}. In the present paper we study functionals of non-Poisson point processes given by a density w.r.t. a Poisson process. Specially $U$-statistics are of interest and general formulas for their moments are given based on conditional intensities. The paper yields an alternative approach to the moment evaluation given by \cite{RefD} where it is based on Georgii-Nguyen-Zessin formula. The product of a functional and a density is further studied in a logarithmic form using the characterization theorem for Gibbs processes from \cite{RefA}. There is an inequality between moments of Poisson and non-Poisson functionals in this case, and we have a version of central limit theorem in the Poisson case.

In the second part of the paper parametric models for point processes of interacting particles \cite{RefM} are investigated as a special case of the general theory. We concentrate on lower-dimensional particles, namely interacting segments in the plane and plates in the three-dimensional space and their natural $U$-statistics. Mixed moments moments are presented in a closed form using explicit formulas or by means of partitions. Limitations on the parameter space are indicated. Finally in the Poisson case using results from \cite{RefL} the central limit theorem for a vector of $U$-statistics of the model is discussed.

\section{Moments of functionals of point processes having a density}\label{sec:1} 
Consider a bounded Borel set $B\subset {\mathbb R}^d$ with Lebesgue measure $|B|>0$ and a measurable space $({\mathbf N},{\mathcal N})$ of integer-valued finite measures on $B.$ ${\mathcal N}$ is the smallest $\sigma $-algebra which makes the mappings $x\mapsto x(A)$ measurable for all Borel sets $A\subset B$ and all $x\in {\mathbf N}.$ A random element having a.s. values in $({\mathbf N},{\mathcal N})$ is called a finite point process. Let a Poisson point process $\eta $ on $B$ have finite intensity measure $\lambda $ with no atoms and distribution $P_\eta $ on $\mathcal N.$ We consider a finite point process $\mu $ on $B$ given by a density $p$ w.r.t. $\eta ,$ i.e. with distribution $P_\mu $ \begin{equation}\label{dns}dP_\mu (x)=p(x)dP_\eta (x),\; x\in \mathbf N,\end{equation} where $p:{\mathbf N}\rightarrow {\mathbb R}_+$ is measurable satisfying $$\int_{\mathbf N}p(x)dP_\eta (x)=1.$$ For a measurable map $F:{\mathbf N}\rightarrow {\mathbb R},$ $F(\mu )$ is a random variable.
As described in \cite{RefA}, p.61, integer-valued finite measures can be represented in this context by $n$-tuples of points corresponding to their support ($n$ is variable). Sometimes we will apply this representation without using its explicit notation from \cite{RefA}. We deal with $L_p$ spaces, $1\leq p<+\infty ,$ of functions on various measure spaces. The objective of the present paper is formula $${\mathbb E}F(\mu )={\mathbb E}[F(\eta )p(\eta )].$$ 
\begin{lemma}\label{stva}
Let for fixed $m\in {\mathbb N}$ it holds $F\in  L_m(P_\mu ),\; G_m(x)=F^m(x)p(x).$ Then the $m$-th moment \begin{equation}\label{mom}{\mathbb E}F^m(\mu )={\mathbb E}G_m(\eta ),\end{equation} specially for $m=1,2$ we have \begin{equation}\label{stro}
{\mathbb E}F(\mu )={\mathbb E}G_1(\eta ),\quad var F(\mu )={\mathbb E}G_2(\eta )-[{\mathbb E}G_1(\eta )]^2.\end{equation}
\end{lemma}
\noindent{\bf Proof:} It holds ${\mathbb E}F^m(\mu )=\int F^m(x)dP_\mu (x)=\int F^m(x)p(x)dP_\eta (x)={\mathbb E}G_m(\eta ),$ specially ${\mathbb E}F(\mu )={\mathbb E}G_1(\eta ),\; var F(\mu )={\mathbb E}F(\mu )^2-({\mathbb E}F(\mu ))^2$
\hfill $\Box $\\

For a functional $F,\;y\in B,$ one defines the difference operator $D_yF$ for a point process $\mu $ as a random variable $$D_yF(\mu )=F(\mu + \delta_y)-F(\mu ),$$ where $\delta_y $ is a Dirac measure at the point $y.$ Inductively for $n\geq 2$ and $(y_1,\dots ,y_n)\in B^n$ we define a function $$D_{y_1,\dots ,y_n}^nF=D_{y_1}^1D_{y_2,\dots ,y_n}^{n-1}F,$$
where $D^1_y=D_y,\; D^0F=F.$ Operator $D^n_{y_1,\dots ,y_n}$ is symmetric in $y_1,\dots ,y_n$ and symmetric functions $T_n^\mu F$ on $B^n$ are defined as $$T^\mu_nF(y_1,\dots ,y_n)={\mathbb E}D_{y_1,\dots ,y_n}^nF(\mu ),$$ $n\in {\mathbb N},\; T_0^\mu F={\mathbb E}F(\mu ),$ whenever the expectations exist. We write $T_nF$ for $T^\eta_nF.$

For the functionals of a Poisson process Theorem 1.1 in \cite{RefJ} says that given $F,\tilde{F}\in L^2(P_\eta )$ it holds
\begin{equation}\label{skaso}{\mathbb E}[F(\eta ){\tilde F}(\eta )]={\mathbb E}F(\eta ){\mathbb E}{\tilde F}(\eta )+\sum_{n=1}^\infty\frac{1}{n!}\langle T_nF,T_n{\tilde F}\rangle_n,\end{equation} where $\langle .,.\rangle_n$ is the scalar product in $L_2(\lambda^n).$
 
\subsection{Explicit formulas for $U$-statistics}
\label{sec:2}
A $U$-statistic of order $k\in {\mathbb N}$ of a finite point process $\mu $ is a functional defined by
\begin{equation}\label{ust}F(\mu )=\sum_{(x_1,\dots ,x_k)\in\mu^k_{\neq }} f(x_1,\dots ,x_k),\end{equation} where $f:B^k\rightarrow {\mathbb R}$ is a function symmetric w.r.t. to the permutations of its variables, $f\in L_1(\lambda^k ).$ Here $\mu^k_{\neq }$ is the set of $k$-tuples of different points of $\mu .$ We say that $F$ is driven by $f.$ By the Slivnyak-Mecke theorem \cite{RefS} we have $${\mathbb E}F(\eta )=\int_B\dots\int_Bf(x_1,\dots ,x_k)\lambda (d(x_1,\dots ,x_k)),$$ where we write $\lambda (d(x_1,\dots ,x_k))$ instead of $\lambda (dx_1)\dots\lambda (dx_k).$ This notation is used throughout the whole paper. Following \cite{RefB} for $F\in L_2(P_\eta )$ using (\ref{skaso}) it holds \begin{equation}\label{vari}var F(\eta )=\sum_{i=1}^ki!\binom{k}{i}^2\times\end{equation}$$\int_{B^i}\left(\int_{B^{k-i}}f(y_1,\dots ,y_i,x_1,\dots ,x_{k-i})\lambda(d(x_1,\dots ,x_{k-i}))\right)^2\lambda (d(y_1,\dots ,y_i)).$$
It is then derived that for $U$-statistic of order $k$ it holds \begin{equation}\label{ustd}D^n_{y_1,\dots ,y_n}F=\frac{k!}{(k-n)!}\sum_{(x_1,\dots ,x_{k-n})\in\mu^{k-n}_{\neq}} f(y_1,\dots ,y_n,x_1,\dots ,x_{k-n})\end{equation} for $n\leq k,\;D^n_{y_1,\dots ,y_n}F=0$ for $n>k.$
Thus \begin{equation}\label{tnf}T_nF(y_1,\dots ,y_n)=\frac{k!}{(k-n)!}\int_{B^{k-n}} f(y_1,\dots ,y_n,x_1,\dots ,x_{k-n})\lambda (d(x_1,\dots ,x_{k-n})),\end{equation}$n\leq k,\; T_nF(y_1,\dots ,y_n)=0,n>k.$

Let $\mu $ be a finite point process with density $p$ satisfying \begin{equation}\label{hered}p(x)>0\Rightarrow p(\tilde{x})>0\end{equation} for all $\tilde{x}\subset x.$ For the (Papangelou) conditional intensity of $\mu ,$ see \cite{RefA}, it holds $$\lambda^*(u,x)=\frac{p(x\cup \{u\})}{p(x)},\; x\in {\mathbf N},\: u\in B,\: u\notin x,$$ here probability $P(u\in\mu )=0.$ For $p(x)=0$ we put $\lambda^*(u,x)=0.$ For $n>1$ we use analogously a.s. $$\lambda^*_n(u_1,\dots ,u_n,x)=\frac{p(x\cup \{u_1,\dots ,u_n\})}{p(x)},$$ $u_1,\dots ,u_n\in B$ distinct, the conditional intensity of $n$-th order of $\mu ,\;\lambda^*_0\equiv 1.$ We observe that $\lambda^*_n$ is symmetric in the variables $u_1,\dots ,u_n.$ A point process $\mu $ with conditional intensity $\lambda^*$ has intensity function \begin{equation}\label{inyt}\rho (u)={\mathbb E}\lambda^*(u,\mu ).\end{equation}
\begin{lemma}\label{pap}Let $p\in L_2(P_\eta ),\;n\in {\mathbb N},$
then $\lambda^n$-a.e. it holds \begin{equation}\label{tnp}T_np(y_1,\dots , y_n)=\sum_{J\subset \{1,\dots ,n\}}(-1)^{n-|J|}{\mathbb E}\lambda^*_{|J|}(\{y_j,\:j\in J\},\mu ),\end{equation} where $|J|$ is the cardinality of $J.$
\end{lemma}
\noindent{\bf Proof:} Under the assumption $p\in L_2(P_\eta )$ it follows from (\ref{skaso}) that $T_np\in L_2(\lambda^n)$ and
since  $D^n_{y_1,\dots ,y_n}p(\eta )=\sum_{J\subset \{1,\dots ,n\}}(-1)^{n-|J|}p(\eta\cup \{y_j,\:j\in J\}),$ cf. \cite{RefJ}, we have $$T_np(y_1,\dots , y_n)={\mathbb E}D^n_{y_1,\dots ,y_n}p(\eta )=$$$$=\int\sum_{J\subset \{1,\dots ,n\}}(-1)^{n-|J|}p(x\cup \{y_j,\:j\in J\})\frac{dP_\mu (x)}{p(x)}$$ $\lambda^n$-a.e. and (\ref{tnp}) follows. \hfill $\Box $\\
\begin{theorem}\label{th1}Let $F_j$ be $U$-statistics of order $k_j,\:j=1,\dots ,m,$ such that $$\prod_{j=1}^mF_j\in L_2(P_\eta )$$ and the density $p\in L_2(P_\eta ).$ Then it holds \begin{equation}\label{mn}{\mathbb E}\left[\prod_{j=1}^mF_j(\mu )\right]={\mathbb E}\left[\prod_{j=1}^mF_j(\eta )\right]+ \sum_{n=1}^q\frac{1}{n!} \langle T_n\prod_{j=1}^mF_j,T_np\rangle_n, \end{equation} where $q={\sum_{i=1}^mk_i}.$ 
\end{theorem}
\noindent{\bf Proof:} Using formula (\ref{skaso}) with ${\mathbb E}p(\eta )=1$ we claim that \begin{equation}\label{nmn}T_n\prod_{j=1}^mF_j=0,\; n>q.\end{equation} For two $U$-statistics $F,G$ of order $k,l$ driven by $f,g,$ respectively, we have $$D_yFG(\eta )= 
\sum_{(x_1,\dots ,x_{k})\in(\eta\cup y)^k_{\neq}} f(x_1,\dots ,x_k)\sum_{(z_1,\dots ,z_l)\in(\eta\cup y)^k_{\neq}}g(z_1,\dots z_l)-$$$$-\sum_{(x_1,\dots ,x_{k})\in\eta^k_{\neq}} f(x_1,\dots ,x_k)\sum_{(z_1,\dots ,z_l)\in\eta^k_{\neq}}g(z_1,\dots z_l).$$ Only terms where $y$ is among variables (either in one or both sums) in the first product on the right side do not cancel with any term in the second product. Thus for the second difference there is one place less for variables (since $y$ is fixed). After $k+l$ differences all places are occupied and $D^{k+l}_{y_1,\dots ,y_{k+l}}$ is independent of the Poisson process. Therefore the $(k+l+1)$-st difference is zero and (\ref{nmn}) holds for a product of two functionals. From the same reasoning with more than two $U$-statistics (\ref{mn}) follows.
\hfill $\Box $\\
\begin{theorem}\label{theo1}
For a $U$-statistic $F\in L_2(P_\eta )$ of order $k$ and density $p\in L_2(P_\eta )$ it holds \begin{equation}\label{effa}{\mathbb E}F(\mu )=\int_{B^k}f(x_1,\dots ,x_k){\mathbb E}[\lambda^*_k(x_1,\dots ,x_k,\,\mu)]\lambda (d(x_1,\dots ,x_k)). \end{equation}  
\end{theorem}
\noindent{\bf Proof:} Denote $C_j^n$ the set of all combinations $c=\{ c_1,\dots ,c_j\}$ of distinct numbers from $\{1,\dots ,n\}.$
We put (\ref{tnf}) and (\ref{tnp}) into (\ref{mn}) with $m=1$ and obtain $${\mathbb E}F(\mu )= \sum_{n=0}^{k}\frac{1}{n!}\int_{B^n}\frac{k!}{(k-n)!}\times $$$$\int_{B^{k-n}} f(y_1,\dots ,y_n,x_1,\dots ,x_{k-n})\lambda (d(x_1,\dots ,x_{k-n}))\times$$$$\sum_{j=0}^n(-1)^{n-j}\sum_{c\in C_j^n} {\mathbb E}\lambda_j^*(y_{c_1},\dots ,y_{c_j},\mu )\lambda (d(y_1,\dots ,y_n))=$$\begin{equation}\label{bnom}= \sum_{j=0}^{k}\sum_{n=j}^{k}(-1)^{n-j}\binom{k}{n}\times \end{equation}$$\int_{B^{k}}\sum_{c\in C_j^n} {\mathbb E}\lambda_j^*(y_{c_1},\dots ,y_{c_j},\mu )f(y_1,\dots ,y_{k})\lambda (d(y_1,\dots ,y_k)).$$
The cardinality of $C_j^n$ is $\binom{n}{j}$ and the identity $$\sum_{n=j}^k(-1)^{n-j}\binom{k}{n}\binom{n}{j}=0,\; j<k$$ holds, see \cite{RefK}, p.39, identity 11. Thus for each fixed $j<k$ it follows that the inner sum over $n$ in (\ref{bnom}) vanishes, while the remaining value $j=k$ yields the result. \hfill $\Box $

For a function $h\in L_1(\lambda^k)$ not necessarily symmetric, the symmetrization $${\cal S}(h)(x_1,\dots ,x_{k})=\frac{1}{k!}\sum_{q\in{\cal T}_{k}}h(x_{q_1},\dots ,x_{q_{k}}),$$ where ${\cal T}_{k}$ is the set of all permutations of indices $1,\dots ,k,$ is a symmetric function.
We observe that \begin{equation}\label{smm}\sum_{(x_1,\dots ,x_{k})\in\mu^k_{\neq}}h(x_1,\dots ,x_k)=\sum_{(x_1,\dots ,x_{k})\in\mu^k_{\neq}}{\cal S}(h)(x_1,\dots ,x_{k})\end{equation} is a $U$-statistic of order $k.$
\begin{lemma}\label{Lem2} Let $m\in {\mathbb N},\; F_i$ be $U$-statistics of orders $k_i$ driven by functions $f_i,$ respectively, $i=1,\dots ,m,\; k_1\geq k_2\geq \dots \geq k_m.$ Then there exist functions $h_{k_1,j_2,\dots ,j_m}:\mathbb{R}^{k_1+\sum_{i=2}^mj_{i}}\longrightarrow[0,\infty),\; j_{i}=0,\dots,k_i,\; i=2,\dots ,m,$ such that \begin{equation}\label{ust2}\prod_{i=1}^mF_i(\mu )=\end{equation}$$=\sum_{j_{2},\dots,j_{m}}
A_{j_{2}:j_{m}}\sum_{(x_1,\dots ,x_{k_1+\sum_{i=2}^mj_{i}})\in\mu_{\neq}^{k_1+\sum_{i=2}^mj_{i}}}h_{k_1,j_2,\dots ,j_m}(x_1,\dots ,x_{k_1+\sum_{i=2}^mj_{i}})$$ where we sum over $j_{i}=0,\dots,k_i,\; i=2,\dots ,m$ and
\begin{equation}\label{acka}A_{j_{2}:j_{m}}=\prod_{l=2}^m\binom{k_l}{j_l}{k_1!(k_1+j_2)!\dots(k_1+\sum_{i=2}^{m-1}j_{i})!\over(k_1+j_2-k_2)!(k_1+j_2+j_3-k_3)!\dots(k_1+\sum_{i=2}^mj_{i}-k_m)!}.\end{equation}\end{lemma}
{\bf Proof:} We proceed by induction in the number of functions $n=2,\dots ,m.$ For $n=2$ and $U$-statistics $F_1,F_2$ of orders $k_1,k_2$ driven by $f_1,f_2,$ respectively, $k_1\geq k_2,$ we have \begin{equation}\label{sdva}
F_1(\mu )F_2(\mu )=\sum_{j_2=0}^{k_2}\binom{k_2}{j_2}\frac{k_1!}{(k_1-k_2+j_2)!}\times\end{equation}$$\sum_{(x_1,\dots ,x_{k_1+j_2})\in\mu^{k_1+j_2}_{\neq}} f_2(x_1,\dots ,x_{k_2})f_1(x_1,\dots ,x_{k_2-j_2}, x_{k_2+1},\dots ,x_{k_1+j_2}),$$
since the product $F_1F_2$ of $U$-statistics is a sum of $k_2+1$ terms, which are sums (over $k_2+j_2$ distinct points from $\mu $) of products $f_2(x_1,\dots ,x_{k_2})f_1(y_1,\dots ,y_{k_1}),$ where $k_2-j_2$ variables appear simultaneously in both lists of variables of the product, $j_2=0,1,\dots ,k_2.$ Their first occurence is independent of the order (since all orders are present in the inner sum of (\ref{sdva})) while their second occurence is dependent on the order. Therefore coefficients at the inner sums are equal to $$\binom{k_2}{j_2}\binom{k_1}{k_1-k_2+j_2}(k_2-j_2)!,\;j_2=0,1,\dots ,k_1,$$ and denoting $$h_{k_1,j}(x_1,\dots ,x_{k_1+j})=f_2(x_1,\dots ,x_{k_2})f_1(x_1,\dots ,x_{k_2-j}, x_{k_2+1},\dots ,x_{k_1+j})$$ leads to the result for $n=2.$
We can use the symmetrization argument (\ref{smm}) to claim that the inner sum (\ref{sdva}) is a $U$-statistic for each $j_2=0,\dots ,k_2.$ 
Further let (\ref{ust2}) and (\ref{acka}) hold for $m-1$ and we consider the product $$\prod_{i=1}^{m-1}F_i(\mu )F_m(\mu ).$$ We have $k_m\leq k_1+\sum_{i=2}^ {m-1}j_i$ for any $j_i=0,\dots ,k_i,\; i=2,\dots ,m-1,$ so using the same argument as above in the case $n=2$ to any term in the outer sum of $\prod_{i=1}^{m-1}F_i(\mu )$ when multiplied by $F_m(\mu )$ the induction step is finished. \hfill $\Box $\\
 
\noindent {\bf Remark 1} Lemma \ref{Lem2} shows how to compute coefficients at the terms of the product explicitly. Instead of trying to express functions $h_{k_1,j_2,\dots ,j_{m}}$ by means of functions $f_i$ we can use a short expression given by diagrams \cite{RefT}, \cite{RefL}. Let $[k]=\{1,\dots ,k\},$ denote $\Pi_k$ the set of all  partitions $\{J_i\}$ of $[k],$ where $J_i$ are disjoint blocks and $\cup J_i=[k].$ For
$k=k_1+\dots +k_m$ and blocks $$J_i=\{ j: k_1+\dots +k_{i-1}< j\leq k_1+\dots +k_i\},\; i=1,\dots ,m,$$ consider the partition $\pi =\{J_i,\;1\leq i\leq m\}$ and let
$\Pi_{k_1,\dots ,k_m}\subset \Pi_k$ be the set of all partitions $\sigma\in\Pi_k$ such that $|J\cap J'|\leq 1$ for all $J\in\pi $ and all $J'\in\sigma .$ Here $|J|$ is the cardinality of a block $J\in\sigma .$ 
For a partition $\sigma\in\Pi_{k_1\dots k_m}$ we define the function $(\otimes_{j=1}^mf_j)_\sigma:B^{|\sigma |}\rightarrow {\mathbb R}$ by replacing all variables of the tensor product $\otimes_{j=1}^mf_j$ that belong to the same block of $\sigma $ by a new common variable, $|\sigma |$ is the number of blocks in $\sigma .$
 Under the assumptions of Lemma \ref{Lem2} we have \begin{equation}\label{divis}\prod_{i=1}^mF_i(\mu )=\sum_{\sigma\in\Pi_{k_1\dots k_m}}\sum_{(x_1,\dots ,x_{|\sigma |})\in\mu_{\neq}^{|\sigma |}}(\otimes_{i=1}^m f_i)_{|\sigma |}(x_1,\dots ,x_{|\sigma |}).\end{equation} This is demonstrated by the fact that $\sum_{j_{2},\dots,j_{m}}
A_{j_{2}:j_{m}}=card \prod_{k_1\dots k_m}$ which is proved by induction in $m,$ for $m=1$ we have $card \prod_{k_1}=1.$ Induction step $m-1\rightarrow m$ follows since for a new block $J_m\in\pi $ with cardinality $k_m$ and $0\leq j_m\leq k_m$ the term $\binom{k_m}{j_m}$ yields the number of combinations of $j_m$ blocks $J$ of partitions $\sigma\in\prod_{k_1\dots k_m}$ with $|J|=1$ (subsets of $J_m$) and the term $$\frac{(k_1+\sum_{i=2}^{m-1}j_{i})!}{(k_1+\sum_{i=2}^mj_{i}-k_m)!}$$ contributes to the number of partitions $\sigma\in\prod_{k_1\dots k_m}$ when the remaining $k_m-j_m$ items in $J_m$ participate in blocks with $|J|\geq 2.$ 
\begin{theorem}\label{theo2}Let $m\in\mathbb{N},\;\prod_{i=1}^mF_i\in L_2(P_{\eta}),$ $p\in L_2(P_\eta),$ where $F_i$ are $U$-statistics of orders $k_i$ driven by nonnegative functions $f_i,$ respectively, $i=1,\dots,m$. Then \begin{equation}\label{divi}{\mathbb E}\prod_{i=1}^mF_i(\mu )=\sum_{\sigma\in\prod_{k_1\dots k_m}}\int_{B^{|\sigma |}}(\otimes_{i=1}^mf_i)_{|\sigma |}(x_1,\dots ,x_{|\sigma |})\times \end{equation}$$\times {\mathbb E}\lambda^*_{|\sigma |}(x_1,\dots ,x_{|\sigma |};\mu )\lambda(d(x_1,\dots ,x_{|\sigma |}))$$\end{theorem}
 \noindent{\bf Proof:} In formula (\ref{divis}) each term \begin{equation}\label{ftl}\sum_{(x_1,\dots ,x_{|\sigma |})\in\mu_{\neq}^{|\sigma |}}(\otimes_{i=1}^m f_i)_{|\sigma |}(x_1,\dots ,x_{|\sigma |})=\end{equation}$=\sum_{(x_1,\dots ,x_{|\sigma |})\in\mu_{\neq}^{|\sigma |}}{\cal S}((\otimes_{i=1}^m f_i)_{|\sigma |})(x_1,\dots ,x_{|\sigma |})$ is a $U$-statistic by symmetrization. If we square formula (\ref{divis}) with $\eta $ instead of $\mu ,$ the expectation of right hand side is finite, which sums only nonnegative terms and involves squares of the inner sums of (\ref{divis}). Therefore each corresponding functional belongs to $L_2(P_{\eta}),$ we can apply Theorem \ref{theo1} to all inner sums of (\ref{divis}) from which the result follows. \hfill $\Box $\\ 
{\bf Remark 2} 
 Specially we have for $m=2:$ \begin{equation}\label{prodva}{\mathbb E}[F_1(\mu )F_2(\mu )]=\sum_{j=0}^{k_2}\binom{k_2}{j}\frac{k_1!}{(k_1-k_2+j)!}\times\end{equation}
$$\times\int_{B^{k_1+j}}h_{k_1,j}(x_1,\dots ,x_{k_1+j}){\mathbb E}[\lambda^*_{k_1+j}(x_1,\dots ,x_{k_1+j},\,\mu)]\lambda (d(x_1,\dots ,x_{k_1+j})).$$ Formula (\ref{divi}) has an analogous structure (including higher-order conditional intensities) as the formula in Proposition 3.1 in \cite{RefD} (derived from Georgii-Nguyen-Zessin formula), where the integrated functions have a simpler form. While this cited paper has a more general background, our present paper is directed to explicit results for $U$-statistics and applications in stochastic geometry.

The assumptions of the above Theorems can be verified using formula for the expectation of a nonnegative functional of a Poisson process, see \cite{RefW}, p.15: \begin{equation} \label{odhad} {\mathbb E}[F(\eta )]=e^{-\lambda (B)}\sum_{n=0}^\infty\frac{1}{n!}\int_B\dots\int_B F(u_1,\dots ,u_n)\lambda (d(u_1,\dots ,u_n)). \end{equation} 
\begin{example}\label{ex2}
Consider $k=1,\; C\subset B$ measurable and $U$-statistic $$F(\mu )=\sum_{y\in\mu }f(y)=\mu (C),\quad f(y)=\,1_{[y\in C]}.$$
Let $\beta >0,\; 0\leq\gamma\leq 1,\; r>0$ be parameters, $\mu $ a Strauss point process \cite{RefA} on $B\subset {\mathbb R}^d$ bounded 
with density \begin{equation}\label{exy} p(x)=\alpha\beta^{n(x)}\gamma^{s(x)},\quad s(x)=\sum_{y,z\in x^2_{\neq }}\,1_{[||z-y||\leq r]},\end{equation} w.r.t. the Poisson point process with Lebesgue intensity measure $\lambda ,\;\alpha $ is the normalizing constant, $n(x)$ the number of points in $x.$ Here conditional intensities $$\lambda^*(u,x)=\beta\gamma^{t(u,x)},\quad \lambda^*_2(y_1,y_2,x)=\beta^2\gamma^{1_{[||y_1-y_2||\leq r]}}\gamma^{t(y_1,x)+t(y_2,x)},$$ where $t(u,x)=\sum_{y\in x}\,1_{[||u-y||\leq r]}.$ The assumptions of
Theorems \ref{theo1} and \ref{theo2} are verified using (\ref{odhad}), since e.g. $p^2(x)\leq\alpha^2\beta^{2n(x)}$ and $\sum_{n=0}^\infty \frac{\beta^{2n}\lambda(B)^n}{n!}<\infty ,$ analogously for $F^2,\: F^4.$ Thus we obtain $${\mathbb E}\mu (C)=\beta\int_C {\mathbb E}[\gamma^{t(y,\mu )}]\lambda (dy),$$$${\mathbb E} [\mu (C)^2]=\beta\int_C {\mathbb E}[\gamma^{t(y,\mu )}]\lambda (dy)+$$$$+\beta^2\int_C\int_C\gamma^{1_{[||y_1-y_2||\leq r]}}{\mathbb E}[\gamma^{t(y_1,\mu )+t(y_2,\mu )}]\lambda(d(y_1,y_2)).$$
\end{example}
\begin{example}\label{ex3}
The special case of Strauss process with $\gamma =1$ in (\ref{exy}) is Poisson process $\eta_\beta $ with deterministic constant conditional intensities $\lambda_n^*(u,\eta_\beta )=\beta^n,$ $n=1,2,\dots $ and constant intensity function $\beta ,$ cf. (\ref{inyt}). An easy exercise is to verify that formula (\ref{vari}) for $\eta_\beta $ is a special case of (\ref{prodva}).
\end{example}
\subsection{Functionals in logaritmic form}
In Lemma \ref{stva} we used the relation $${\mathbb E}F^m(\mu )={\mathbb E}[F^m(\eta )p(\eta )],\; m=1,2,\dots ,$$ where $\eta $ is a Poisson process and $\mu $ a point process with probability density $p$ w.r.t. $\eta .$ Consider a functional on $\mathbf N$ \begin{equation}\label{gk}H_m=\log (F^mp)=m\log F+\log p,\; m=1,2,\dots \end{equation} under the assumption $H_m\in L_1(P_\eta ).$ From Jensen inequality we have \begin{equation}\label{jens}\log {\mathbb E}F^m(\mu )\geq {\mathbb E}H_m(\eta ).\end{equation}
 
According to Theorem 4.3 in \cite{RefA} $\lambda^*(u,x),\:x\in {\mathbf N},\: u\in B,$ is a conditional intensity of a point process $\mu $ satisfying (\ref{hered}) if and only if it can be expressed in the form \begin{equation}\label{cig}
\lambda^*(u,x)=\exp \left[V_1(u)+\sum_{y\in x}V_2(u,y)+\sum_{(y_1,y_2)\in x^2_{\neq }}V_3(u,y_1,y_2)+\dots\right],\end{equation}
where $V_k:B^k\rightarrow {\mathbb R}\cup \{-\infty\} $ is called the potential of order $k.$ Then the density is that of a Gibbs process \begin{equation}\label{gid}p(x)=\exp\left[
V_0+\sum_{y\in x}V_1(y)+\sum_{(y_1,y_2)\in x^2_{\neq }}V_2(y_1,y_2)+\dots\right].\end{equation}  
Consequently $$\log p(x)=V_0+\sum_{y\in x}V_1(y)+\sum_{(y_1,y_2)\in x^2_{\neq }\subset x}V_2(y_1,y_2)+\dots $$ is a sum of a constant and $U$-statistics. 

Assume that there is only a finite number $l$ of sums on the right side of (\ref{cig}) and further that \begin{equation}\label{prf}F(\eta )=\exp\left[\sum_{(x_1,\dots ,x_k)\in\eta^k_{\neq }} f(x_1,\dots ,x_k)\right].\end{equation} Then $\log F$ is a $U$-statistic of order $k$ and $H_m$ is a finite sum of $U$-statistics.

\section{Stochastic geometry functionals}
Let $B\subset {\mathbb R}^l,\,l\in {\mathbb N}$ be a bounded Borel set with positive Lebesgue measure, $X$ a germ-grain process \cite{RefS} of germs $z\in B$ and compact grains $K_z\subset {\mathbb R}^l,$ typically $z\in K_z.$ For a realization $x$ of the germ-grain process denote $U_x$ the union of all grains. Consider a probability density \cite{RefW} \begin{equation}\label{den}p(x)=c_\nu ^{-1}\exp(\nu G(U_x)), \end{equation}of $X$
w.r.t. a given reference Poisson point process $\eta .$ Here $\nu =(\nu_1,\dots ,\nu_d)$ is a vector of real parameters, $c_\nu $ a normalizing constant, $G(U_x)\in {\mathbb R}^d$ is a vector of geometrical characteristics of $U_x.$ In the exponent of (\ref{den}) there is the inner (scalar) product in ${\mathbb R}^d.$
The largest set of ${\nu}$ such that exponential family density (\ref{den}) is well defined is
$\{\nu\in\mathbb{R}^d : \mathbb{E}[\exp(\nu G(U_\eta ))] <\infty\},$ see \cite{RefW}.
For a vector of geometrical characteristics $G(U_x)=(G_1({U_x}),\dots,G_r({U_x})), r\in\mathbb{N}$
denote $$D^m_{y_1,\dots,y_m}G(U_x)=(D^m_{y_1,\dots,y_m}G_1(U_{x}),\dots,D^m_{y_1,\dots,y_m}G_r(U_{x}))^T$$
the vector of $m-$th differences.
\begin{theorem}\label{Th:papan}Consider the probability density (\ref{den}).
 Then for the corresponding Papangelou conditional
intensity $\lambda^*_m$ of order $m\in {\mathbb N}$ and $x\in{\mathbf N}$ it holds
\begin{equation}
\lambda^*_m(y_m,\dots,y_1,x)=e^{\nu Q_mG(U_x)}\qquad
a.s.,
\end{equation}
where
\begin{eqnarray*}
Q_mG(U_x)&=&D^m_{y_1,\dots,y_m}G(U_x)\\
&&+\sum_{i_1,\dots,i_{m-1}\in\{1,\dots,m\}}D^{m-1}_{y_{i_1},\dots,y_{i_{m-1}}}G(U_x)+\dots+
\sum_{1\leq i\leq m}D_{y_i}G(U_x).
\end{eqnarray*}
\end{theorem}
\begin{proof}
We have for $x\in{\mathbf N}$
$$\lambda^*_m(y_1,\dots,y_m,x)={p_{\nu}(x\cup\{y_1,\dots,y_m\})\over p_{\nu}(x)}=e^{\nu G({U_x\cup\{y_1,\dots,y_m\}})-\nu G(U_x)}.$$
We need to prove that
\begin{equation}\label{kvm}Q_mG(U_x)=G({U_x\cup\{y_1,\dots,y_m\}})-G(U_x).\end{equation}
For $m=1$ we have
$$\lambda^{*}_1(y;x)={e^{\nu G({U_x\cup \{y\}})}\over e^{\nu
G(U_x)}}=e^{\nu(G({U_x\cup
\{y\}})-G(U_x))}=e^{\nu D^1_y
G(U_x)}=e^{\nu Q_1G(U_x)}.$$
Now assume that the formula (\ref{kvm}) holds for $m-1$ and we shall prove it for $m.$ Firstly split $Q_mG(U_x):$
\begin{equation}\label{kvmn}
Q_mG(U_x)=D^1_{y_m}G(U_x)+\end{equation}$$
+\sum_{j=1}^{m-1}D^2_{y_j,y_m}G(U_x)+\sum_{1\leq
i<j\leq m-1}D^3_{y_i,y_j,y_m}G(U_x)+\dots+D^m_{y_1,\dots,y_m}G(U_x)+$$$$
+\sum_{j=1}^{m-1}D^1_{y_j}G(U_x)+\sum_{1\leq
i<j\leq
m-1}D^2_{y_i,y_j}G(U_x)+\dots+D^{m-1}_{y_1,\dots,y_{m-1}}G(U_x).
$$
From the assumption the third line of (\ref{kvmn}) is equal to $Q_{m-1}G(U_x)$ and further
$$
Q_mG(U_x)=G({U_x\cup\{y_1,\dots,y_{m-1}\}})-G(U_x)+D^1_{y_m}G(U_x)+$$$$
+D^1_{y_m}\biggl(\sum_{j=1}^{m-1}D^1_{y_j}G(U_x)+\sum_{1\leq i<j\leq
m-1}D^2_{y_i,y_j}G(U_x)+\dots+D^{m-1}_{y_1,\dots,y_{m-1}}G(U_x)\biggr)
=$$$$
=G({U_x\cup\{y_1,\dots,y_{m-1}\}})-G(U_x)+D^1_{y_m}G(U_x)+$$$$+D^1_{y_m}\bigl(G({U_x\cup\{y_1,\dots,y_{m-1}\}})-G(U_x)\bigr)$$$$
=G({U_x\cup\{y_1,\dots,y_{m-1}\}})-G(U_x)+D^1_{y_m}G(U_x)+G({U_x\cup\{y_1,\dots,y_m\}})-$$$$-G({U_x\cup\{y_1,\dots,y_{m-1}\}})-D^1_{y_m}G(U_x)
=G({U_x\cup\{y_1,\dots,y_m\}})-G(U_x).
$$
\end{proof}
\subsection{Particular models}
The intensity of the reference process depends on a specific model, see \cite{RefM} for interacting discs. Here we consider process of interacting segments in ${\mathbb R}^2$ or interacting plates in ${\mathbb R}^3$ where we study natural $U$-statistics. Consider first $B\subset {\mathbb R}^2,$ \begin{equation}\label{bck}Y=B\times (0,b]\times [0,\pi ),\end{equation} where $b>0$ is an upper bound for the segment length. The Poisson process $\eta $ on $Y$ has intensity measure $\lambda ,$ \begin{equation}\label{inti}\lambda (d(z, r, \phi ))=\rho (z)dzQ(dr)V(d\phi ),\end{equation} where $z$ denotes the location of the segment centre, $r$ the segment length and $\phi $ its axial orientation, $Q,V$ are probability measures, $V$ nondegenerate, $\rho $ a bounded intensity function of germs on $B.$ The segment process $\mu $ has the density (\ref{den}) with $\nu = (\nu_1,\nu_2 ),$ we assume $\nu_2\leq 0$ tu guarantee that $p$ is a probability density.
Further \begin{equation}\label{gseg}G(U_x)=(L(U_x),\: N(U_x)),\end{equation} where $L$ is the total length of all segments and $N$ the total number of intersections between segments. Thus if $l$ is the length of an individual segment
\begin{equation}\label{lsg}L(U_\mu )=\sum_{s\in\mu }l(s)\end{equation} is $U$-statistic of the first order and \begin{equation}\label{nsg}N(U_\mu )=\frac{1}{2}\sum_{(s,t)\in\mu^2_{\neq } }1_{[s\cap t\neq\emptyset ]}\end{equation} is $U$-statistic of the second order. 

Similarly we consider $B\subset {\mathbb R}^3$ and a Poisson process $\eta $ in \begin{equation}\label{bcp}Y=B\times (0,b]\times {\mathbb S}^2,\end{equation} where $b>0$ is an upper bound for the plate radius and ${\mathbb S}^2$ is the unit hemisphere in ${\mathbb R}^3,$ with intensity measure $\lambda $ on $Y$ $$\lambda (d(z, r, \phi ))=\rho (z)dzQ(dr)V(d\phi )$$ where $z$ denotes the location of circular plate centre, $r$ the radius of the plate and $\phi $ its normal orientation.
The point process $\mu $ of circular plates has the density (\ref{den}) w.r.t. $\eta $ with $\nu = (\nu_1,\nu_2,\nu_3 ),$ we assume $\nu_2\leq 0,\;\nu_3\leq 0.$
Further
$$G(U_x)=(S(U_x),\: L(U_x),\: N(U_x)),$$ where $S$ is the total area of plates, $L$ the total length of intersection lines and $N$ the total number of intersection points of triplets of plates. Let $A$ be the area of a single plate, $l$ the length of a single intersection segment, we define
$$S(U_\mu )=\sum_{s\in\mu }A(s),\quad L(U_\mu )=\frac{1}{2}\sum_{(s,t)\in\mu^2_{\neq } }l(s\cap t)$$ which are $U$-statistics of the first, second order, respectively, and $$N(U_\mu )=\frac{1}{6}\sum_{(s,t,u)\in\mu^3_{\neq }
 }1_{[(s\cap t\cap u)\neq\emptyset ]}$$ is $U$-statistic of the third order. 
 
In the following we obtain formulas for the moments of these functionals defined for segment and plate processes. Consider the plate process, for $x\in {\mathbf N},\:y,y_i\in Y,\;y,y_i\notin x$ we have 
$$D_yG(U_x)=\left(\begin{array}{c}D_yS(U_x )\\D_yL(U_x )\\D_yN(U_x )\end{array}\right)=\left(\begin{array}{c}A(y)\\\sum_{s\in x }l(s\cap y)\\\sum_{s,t\in x^2_{\neq }}1_{[s\cap t\cap y\neq\emptyset ]}\end{array}\right),$$$$D^2_{y_1y_2}G(U_x)=\left(\begin{array}{c}0\\l(y_1\cap y_2)\\\sum_{s\in\mu }1_{[s\cap y_1\cap y_2\neq\emptyset ]}\end{array}\right),$$$$D^3_{y_1y_2y_3}G(U_x)=\left(\begin{array}{c}0\\0\\1_{[y_1\cap y_2\cap y_3\neq\emptyset ]}\end{array}\right),$$ higher order diferences are equal to zero.
Denote ${\cal E}_\mu (y)=\exp (\nu D_yG(U_\mu )),$$${\cal E}_\mu (y_1,y_2)=\exp (\nu (D_{y_1}G(U_\mu )+D_{y_2}G(U_\mu )+D^2_{y_1y_2}G(U_\mu ))),$$${\cal E}_\mu (y_1,\dots ,y_m)=\exp (\nu Q_m),\;m=3,...$ where $$Q_m=\sum_{i=1}^m
D_{y_i}G(U_\mu )+\sum_{1\leq i<j\leq m}D^2_{y_iy_j}G(U_\mu )+\sum_{1\leq i<j<l\leq m}D^3_{y_iy_jy_l}G(U_\mu ).$$
From Theorem 4 we obtain the following.
\begin{corollary} For $n\in {\mathbb N}$ we have for the plate process $\mu $ with density (\ref{den}), $\nu_2\leq 0,\;\nu_3\leq 0,$ the conditional intensity of order $n$ $$\lambda_n^*(y_1,\dots ,y_n,\mu )={\cal E}_\mu (y_1,\dots ,y_n)\; a.s. $$
\end{corollary} 
\begin{theorem}\label{theo5}Let $\mu $ be the process of circular plates on $Y$ (\ref{bcp}) with density (\ref{den}), $\nu_2\leq 0,\;\nu_3\leq 0.$
Then $${\mathbb E}S(U_\mu )=\int_Y{\mathbb E}[{\cal E}_\mu (y)]A(y)\lambda (dy),$$$${\mathbb E}L(U_\mu )=\frac{1}{2}\int_{Y^2}{\mathbb E}[{\cal E}_\mu (y_1,y_2)]l(y_1\cap y_2)\lambda (d(y_1,y_2)),$$$${\mathbb E}N(U_\mu )=\frac{1}{6}\int_{Y^3}{\mathbb E}[{\cal E}_\mu (y_1,y_2,y_3)]1_{[y_1\cap y_2\cap y_3\neq\emptyset ]}\lambda (d(y_1,y_2,y_3)).$$
$${\mathbb E}[S(U_\mu )^2]=\int_{Y^2}{\mathbb E}[{\cal E}_\mu (y_1,y_2)]A(y_1)A(y_2)\lambda (d(y_1,y_2))+$$$$+\int_Y{\mathbb E}[{\cal E}_\mu (y)]A(y)^2\lambda (dy),$$
${\mathbb E}[L(U_\mu )^2]=$$$\frac{1}{4}\int_{Y^4}{\mathbb E}[{\cal E}_\mu (y_1,y_2,y_3,y_4)]l(y_1\cap y_2)l(y_3\cap y_4)\lambda (d(y_1,\dots ,y_4))$$$$+\int_{Y^3}{\mathbb E}[{\cal E}_\mu (y_1,y_2,y_3)]l(y_1\cap y_2)l(y_3\cap y_1)\lambda (d(y_1,y_2,y_3))+$$$$+\frac{1}{2}\int_{Y^2}{\mathbb E}[{\cal E}_\mu (y_1,y_2)]l(y_1\cap y_2)^2\lambda (d(y_1,y_2)).$$
${\mathbb E}[N(U_\mu )^2]=$$$\frac{1}{36}\int_{Y^6}{\mathbb E}[{\cal E}_\mu (y_1,\dots ,y_6)]1_{[y_1\cap y_2\cap y_3\neq\emptyset ]}1_{[y_4\cap y_5\cap y_6\neq\emptyset ]}\lambda (d(y_1,\dots ,y_6))+$$$$+\frac{1}{4}\int_{Y^5}{\mathbb E}[{\cal E}_\mu (y_1,\dots ,y_5)]1_{[y_1\cap y_2\cap y_3\neq\emptyset ]}1_{[y_4\cap y_5\cap y_1\neq\emptyset ]}\lambda (d(y_1,\dots ,y_5))+$$$$+\frac{1}{2}\int_{Y^4}{\mathbb E}[{\cal E}_\mu (y_1,\dots ,y_4)]1_{[y_1\cap y_2\cap y_3\neq\emptyset ]}1_{[y_4\cap y_2\cap y_1\neq\emptyset ]}\lambda (d(y_1,\dots ,y_4))+$$$$+\frac{1}{6}\int_{Y^3}{\mathbb E}[{\cal E}_\mu (y_1,y_2,y_3)]1_{[y_1\cap y_2\cap y_3\neq\emptyset ]}\lambda (d(y_1,y_2,y_3)).$$\end{theorem}
\noindent{\bf Proof:} We verify assumptions of Theorems \ref{theo1} and \ref{theo2} from which the formulas follow.
For $x\in\mathbf N$ with $n(x)=n$ we have estimates $$S(U_x)\leq \pi b^2n,\quad L(U_x)\leq 2b\binom{n}{2},\quad N(U_x)\leq \binom{n}{3}.$$ 
Since $\nu_2\leq 0,\;\nu_3\leq 0$ we have $$p^2(x)\leq const.\exp (2\nu_1\pi b^2n(x)),$$ and from (\ref{odhad})$$\sum_{n=0}^\infty\frac{\lambda (Y)^n}{n!}\exp (2\nu_1\pi b^2n)<+\infty .$$ Concerning the powers of $U$-statistics $S(U_x),$ $\: L(U_x),$ $\: N(U_x)$ an analogous estimate of (\ref{odhad}) is finite. \hfill $\Box $\\\\
From Theorem \ref{theo2} one can also obtain explicit formulas for mixed moments of $U$-statistics, e.g.
$${\mathbb E}[L(U_\mu )N(U_\mu )]=\frac{1}{2}\int_{Y^3}{\mathbb E}[{\cal E}_\mu (y_1,y_2,y_3)]l(y_1\cap y_2)1_{[y_1\cap y_2\cap y_3\neq\emptyset ]}\lambda (d(y_1,y_2,y_3))+$$$$+\frac{1}{2}\int_{Y^4}{\mathbb E}[{\cal E}_\mu (y_1,\dots ,y_4)]l(y_1\cap y_2)1_{[y_1\cap y_3\cap y_4\neq\emptyset ]}\lambda (d(y_1,\dots ,y_4))+$$$$+\frac{1}{12}\int_{Y^5}{\mathbb E}[{\cal E}_\mu (y_1,\dots ,y_5)]l(y_1\cap y_2)1_{[y_3\cap y_4\cap y_5\neq\emptyset ]}\lambda (d(y_1,\dots ,y_5)).$$ Higher-order moments can be briefly formulated by formula (\ref{divi}), e.g. $${\mathbb E}[S(U_\mu )L(U_\mu )N(U_\mu )]=\frac{1}{12}\sum_{\sigma\in\prod_{1,2,3}}\int_{Y^{|\sigma |}}(s(.)\otimes l(.\cap .)\otimes 1_{[.\cap .\cap .\neq\emptyset ]})_{|\sigma |}\times $$$$\times {\mathbb E}\lambda^*_{|\sigma |}(x_1,\dots ,x_{|\sigma |};\mu )\lambda^{|\sigma |}(d(x_1,\dots ,x_{|\sigma |})).$$ This expression has ten terms, the coefficients of which can be obtained from (\ref{acka}).

We obtain similar results for the segment process $\mu $ in ${\mathbb R}^2$ with $U$-statistics $G(U_x)$ in (\ref{gseg}). Here we have
for $y,y_i\in Y$ (\ref{bck}), $y,y_i\notin x,\; x\in {\mathbf N}$ $$D_yG(U_x)=\left(\begin{array}{c}l(y)\\\sum_{s\in x}1_{[s\cap y\neq\emptyset ]}\end{array}\right),\;D^2_{y_1y_2}G(U_x)=\left(\begin{array}{c}0\\1_{[y_1\cap y_2\neq\emptyset ]}\end{array}\right).$$ 
Define analogously ${\cal E}_\mu (y)=\exp (\nu D_yG(U_\mu )),\;
{\cal E}_\mu (y_1,\dots ,y_m)=\exp (\nu Q_m),\;m\in {\mathbb N}$
$$Q_m=\sum_{i=1}^m
D_{y_i}G(U_\mu )+\sum_{1\leq i<j\leq m}D^2_{y_iy_j}G(U_\mu ).$$ Observe as in Corollary 1
that a.s. $${\cal E}_\mu(y_1,\dots ,y_m)=\lambda^*_m(y_1,\dots ,y_m),\; m\in {\mathbb N}.$$
\begin{corollary}\label{crdva}Let $\mu $ be the segment process on $Y$ (\ref{bck}) with density (\ref{den}), $\nu_2\leq 0,$ then for $U$-statistics (\ref{lsg}) and (\ref{nsg}) we have $${\mathbb E}L(U_\mu )=\int_Y{\mathbb E}[{\cal E}_\mu (y)]l(y)\lambda (dy),$$$${\mathbb E}N(U_\mu )=\frac{1}{2}\int_{Y^2}{\mathbb E}[{\cal E}_\mu (y_1,y_2)]1_{[y_1\cap y_2\neq\emptyset ]}\lambda (d(y_1,y_2)),$$
$${\mathbb E}[L(U_\mu )^2]=\int_Y{\mathbb E}[{\cal E}_\mu (y)]l(y)^2\lambda (dy)+$$$$+\int_{Y^2}{\mathbb E}[{\cal E}_\mu (y_1,y_2)]l(y_1)l(y_2)\lambda (d(y_1,y_2)),$$\medskip
$${\mathbb E}[N(U_\mu )^2]=\frac{1}{2}\int_{Y^2}{\mathbb E}[{\cal E}_\mu (y_1,y_2)]1_{[y_1\cap y_2\neq\emptyset ]}\lambda (d(y_1,y_2))+$$$$+\int_{Y^3}{\mathbb E}[{\cal E}_\mu (y_1,y_2,y_3)]1_{[y_1\cap y_2\neq\emptyset ]}1_{[y_3\cap y_1\neq\emptyset ]}\lambda (d(y_1,y_2,y_3))+$$$$+\frac{1}{4}\int_{Y^4}{\mathbb E}[{\cal E}_\mu (y_1,y_2,y_3,y_4)]1_{[y_1\cap y_2\neq\emptyset ]}1_{[y_3\cap y_4\neq\emptyset]} \lambda (d(y_1,\dots ,y_4)).$$
\end{corollary} \noindent The proof is as in Theorem \ref{theo5}.
 
The assumptions on the parameter vector $\nu $ correspond to non-attractive interactions among objects (plates or segments). 
\subsection{Geometric functionals in logaritmic form}
Here we deal with $$H_m(\eta )=m\log F(\eta )+\log p(\eta ),\; m=1,2,\dots $$ in (\ref{gk}) having in mind that the process $\mu $ with density $p$ w.r.t. $\eta $ is related by means of $\log {\mathbb E}F^m(\mu )\geq {\mathbb E}H_m(\eta ).$ Now consider the density (\ref{den}) where $$\log p(x)=-\log c_\nu +\nu_1S(U_x)+\nu_2L(U_x)+\nu_3N(U_x)$$ which is a finite Gibbsian form, cf. (\ref{gid}) with $l=3$ non-constant terms. For $F(x)$ consider one of the three choices: $F(x)=e^{S(U_x)},\;e^{L(U_x)},\;e^{N(U_x)},$ accordingly we write $H_m^1, H_m^2, H_m^3,$ respectively:
 \begin{eqnarray*}H_m^1(\eta )&=&-\log c_\nu+(m+\nu_1)S(U_\eta )+\nu_2L(U_\eta )+\nu_3N(U_\eta )\\
H_m^2(\eta )&=&-\log c_\nu+\nu_1S(U_\eta )+(m+\nu_2)L(U_\eta )+\nu_3N(U_\eta )\\
H_m^3(\eta )&=&-\log c_\nu+\nu_1S(U_\eta )+\nu_2L(U_\eta )+(m+\nu_3)N(U_\eta )\end{eqnarray*}
In order to study the statistics $H_m^p$ we need to investigate multivariate behavior of a vector of $U$-statistics, e.g. for the process of plates in ${\mathbb R}^3$
 $$(S(U_\eta ),L(U_\eta ),N(U_\eta )).$$ Generally for $l\geq 1$ and $i=1,\dots ,l$ let $k_i\in {\mathbb N},$ $f^{(i)}\in L_1(\lambda^{k_i})$ be symmetric functions, $$F^{(i)}(\eta )=\sum_{(x_1,\dots ,x_{k_i})\in\eta^{k_i}_{\neq }} f^{(i)}(x_1,\dots ,x_{k_i}).$$ Consider Poisson processes $\eta_a$ with intensity measures $\lambda_a=a\lambda ,\;a>0.$ 
Following \cite{RefL} $U$-statistics $$F^{(i)}_a(\eta_a)=\sum_{(x_1,\dots ,x_{k_i})\in\eta^{k_i}_{a\neq }} f^{(i)}(x_1,\dots ,x_{k_i})$$ are transformed to \begin{equation}\label{tri}
\hat{F}_a^{(i)}=a^{-(k_i-\frac{1}{2})}(F_a^{(i)}-{\mathbb E}F_a^{(i)}).\end{equation}
The asymptotic covariances are
\begin{equation}\label{asco}C_{ij}=\lim_{a\rightarrow\infty }cov(\hat{F}_a^{(i)},\hat{F}_a^{(j)})=\int T_1F^{(i)}(x)T_1F^{(j)}(x) \lambda (dx),\;i,j\in\{1,\dots ,l\}.\end{equation} 
The convergence under the distance between 
 $l$-dimensional random vectors $X,Y$ 
$$d_3 (X,Y)=\sup_{g\in {\cal H}}|{\mathbb E}g(X)-{\mathbb E}g(Y)|,$$ where $\cal H$ is the system of functions $h\in C^3({\mathbb R}^l)$ with $$\max_{1\leq i_1\leq i_2\leq l}\sup_{x\in {\mathbb R}^l}\big|\frac{\partial^2h(x)}{\partial x_{i_1}\partial x_{i_2}}\big|\leq 1,\quad
\max_{1\leq i_1\leq i_2\leq i_3\leq l}\sup_{x\in {\mathbb R}^l}|\frac{\partial^3h(x)}{\partial x_{i_1}\partial x_{i_2}\partial x_{i_3}}|\leq 1$$ implies convergence in distribution. Based on the multi-dimensional Malliavin-Stein inequality derived in \cite{RefZ} for the distance $d_3$ of a random vector from a centered Gaussian random vector $X$ with covariance matrix $C=(C_{ij})_{i,j=1,\dots ,l},$ 
\cite{RefL} show that under the assumption \begin{equation}\label{assd}\int |T_1F^{(i)}|^3d\lambda <\infty,\;i=1,\dots ,l,\end{equation} there exists a constant $c$ such that \begin{equation}\label{mltv}d_3((\hat{F}_a^{(1)},\dots ,\hat{F}_a^{(l)}),X)\leq ca^{-\frac{1}{2}},\; a\geq 1.\end{equation}

\begin{example}Consider the Poisson segment process on $Y$ (\ref{bck}) with intensity measure $\lambda $ (\ref{inti}) and the $U$-statistics (\ref{lsg}) and (\ref{nsg}). In (\ref{asco})
 $$C_{11}=\int_Y l(s)^2\lambda (ds),\;\;   C_{22}=\int_Y\lambda (\{s:s\cap t\neq\emptyset \})^2\lambda (dt),$$$$C_{12}=2\int_Y l(y)\lambda (\{s:s\cap y\neq\emptyset\})\lambda (dy).$$ 
The assumption (\ref{assd}) transforms to conditions: $$\int_Y l(s)^3\lambda (ds)<\infty,\; \int_Y\lambda(\{s;s\cap y\neq\emptyset \})^3\lambda (dy)<\infty .$$ The finiteness of the intensity measure $\lambda $ in (\ref{inti}) and the boundedness of the segments guarantee that all integrals are finite. Thus for the random vector $(\hat{F}^{(1)}_a,\hat{F}^{(2)}_a)$ obtained by transform (\ref{tri}) of $$(L(U_{\eta_a}),N(U_{\eta_a}))$$ both the central limit theorem and the Berry-Esseen type inequality (\ref{mltv}) hold. 
\end{example}
\subsection*{Acknowledgement}
\small
This work was supported by the Czech Science Foundation, grant P201-10-0472 and it is dedicated to the memory of our teacher Professor Josef \v{S}t\v{e}p\'{a}n.

\end{document}